\documentclass[reqno]{amsart}
\usepackage[left=4cm, right=4cm, top=4cm, bottom=4cm]{geometry}
\usepackage[utf8]{inputenc}
\usepackage{mathtools}
\usepackage{amsfonts, amsmath, amssymb}
\usepackage{array}
\usepackage[hidelinks]{hyperref}
\usepackage{mathrsfs}
\hypersetup{colorlinks,breaklinks,
             urlcolor=blue,
             linkcolor=blue}
\usepackage{graphicx}             
\title{Projective Nullstellensatz for not necessarily algebraically closed fields}
\author{Rati Ludhani\vspace{-6mm}}
	\address{Inria Saclay Centre, 1 Rue Honoré d’Estienne d’Orves, 91120 Palaiseau, France}
\email{\href{mailto:ratiludhani@gmail.com}{ratiludhani@gmail.com}}
\thanks{During the course of this work, Rati Ludhani was supported by Prime Minister's Research Fellowship PMRF-192002-256 at IIT Bombay.}

\date{}

\newtheorem{theorem}{Theorem}

\newtheorem{prop}[theorem]{Proposition}
\newtheorem{defn}[theorem]{Definition}
\newtheorem{conjecture}[theorem]{Conjecture}
\newtheorem{cor}[theorem]{Corollary}
\newtheorem{remark}[theorem]{Remark}

\newcommand{{\Fq}}{\mathbb{F}_q}
\newcommand{\F}{\mathbb{F}}
\newcommand{\N}{\mathbb{N}}
\newcommand{\PP}{\mathbb{P}}
\newcommand{\Z}{\mathbb{Z}}
\newcommand{\ZZ}{\mathsf{Z}}
\newcommand{\V}{\mathsf{V}}
\newcommand{\I}{\mathsf{I}}
\newcommand{\A}{\mathbb{A}}
\newcommand{\m}{\mathfrak{m}}
\begin{document}
\maketitle
\centerline{\emph{Dedicated to Professor Sudhir R. Ghorpade on the occasion of his sixtieth birthday.}}

\begin{abstract}
The Nullstellensatz, proved by Hilbert in 1893, is a classical result that holds when the base field is algebraically closed. When the base field is finite, a version of Hilbert's Nullstellensatz is given by Terjanian in 1966. Laksov in 1987 generalized Hilbert's Nullstellensatz to a  $K$-Nullstellensatz when the base field $K$ is not necessarily algebraically closed. However, unlike Tarjanian's Nullstellensatz, Laksov's Nullstellensatz is not very explicit. Later, Laksov and Westin in 1990 proposed a strengthening to Laksov's Nullstellensatz in the form of four conjectures. 

A projective analogue of Nullstellensatz of the classical Nullstellensatz of Hilbert is well-known for projective varieties over algebraically closed fields. For finite fields, the projective analogue of the Nullstellensatz can be derived as an application of Hilbert's Nullstellensatz, though it is not as efficient as Terjanian's Nullstellensatz. Gimenez, Ruano and San-Jos\'e in 2023 strengthened this result by identifying a computationally efficient set. Here, we introduce an even more efficient set that establishes the projective analogue of the Nullstellensatz for finite fields. Additionally, we provide counterexamples to three of the conjectures of Laksov and Westin.
 
\end{abstract}

\section{Introduction}

Let $k$ be a field and let $\bar{k}$ denote an algebraic closure of $k$. Let $K$ be a subfield of $\bar{k}$ containing $k$. 
For a prime power $q$, denote by $\Fq$ the Galois field with $q$ elements. Fix a nonnegative integer $n$. Let $k[X_1,\ldots,X_n]$ (resp. $k[X_0,X_1,\ldots,X_n]$) denote the polynomial ring in $n$ (resp. $n+1$) variables over the field $k$. Hilbert Nullstellensatz, or Hilbert zero-point-theorem, is a classical result of Mathematics that establishes a relation between affine algebraic varieties  and ideals of polynomial rings. 
More precisely, it states that:

\begin{center}
\begin{tabular}{ m{34em} }
  Suppose $K=\bar{k}$ and let $I\subseteq k[X_1,\ldots,X_n]$ be an ideal. Then the vanishing ideal of the set of 
  common zeros of the polynomials of $I$ in $K^n$ is the radical of $I$. 
\end{tabular}
\end{center}

An exemplary proof of this theorem involves first proving its weak form and then applying the Rabinowitsch trick~\cite{R}. The weak form of the Hilbert Nullstellensatz is a straightforward corollary of the Hilbert Nullstellensatz, which states that:
\begin{center}
\begin{tabular}{ m{34em} }
  Suppose $K=\bar{k}$. Then the ideal $I\subseteq k[X_1,\ldots,X_n]$ contains $1$ if and only if the polynomials of $I$ have no common zeros in $K^n$.
\end{tabular}
\end{center}

Significant research has focused on generalizing the Hilbert Nullstellensatz to arbitrary fields and deriving explicit formulas. For fields $K\subsetneq \bar{k}$, i.e., for non-algebraically closed fields, even the weak form may fail to hold. In these cases, the vanishing ideal of the set of common zeros in $K^n$ of the polynomials in the ideal $I$ of $k[X_1,\ldots,X_n]$ can generally be larger than the radical of $I$. In this context, Dubois~\cite{D} and Risler~\cite{Ri} introduced an ideal known as the real radical of $I$ for the Nullstellensatz in the case of real closed fields $K$. Building on this, Adkins, Gianni, and Tognoli~\cite{AGT}, and Laksov~\cite{L} introduced the notion of the $K$-radical, which they used to formulate the Nullstellensatz for arbitrary fields $K$, referring to it as the \emph{Hilbert $K$-Nullstellensatz}. The proof by Laksov~\cite{L} for the Hilbert $K$-Nullstellensatz follows a similar strategy: first proving its weak form and then applying the Rabinowitsch trick~\cite{R}. A much weaker corollary of the Hilbert Nullstellensatz states that:
\begin{center}
\begin{tabular}{ m{34em} }
  Suppose $K=\bar{k}$. If $f\in k[X_1,\ldots,X_n]$ is such that $f({\bf a})=0$ for every ${\bf a}\in K^n$, then $f=0$.
\end{tabular}
\end{center}

We refer to it as the \emph{Very Weak Nullstellensatz}, following Ghorpade~\cite{G}. The terminology may be motivated by the fact that this result holds even when $K$ is an infinite field. However, the Very Weak Nullstellensatz does not hold when $K$ is a finite field. 
Indeed, if $K=\Fq$, the nonzero polynomials $X_1^q-X_1,\ldots, X_n^q-X_n$ in $k[X_1,\ldots,X_n]$ vanish at every point in $K^n$. 
Nonetheless, an explicit formula for the Nullstellensatz over a finite field $K$ is known, thanks to Tarjanian \cite{T}, which states that:
\begin{center}
\begin{tabular}{ m{34em} }
	Suppose $K=\Fq$ and let $I\subseteq k[X_1,\ldots,X_n]$ be an ideal. Then the vanishing ideal    of the set of common zeros of the polynomials of $I$ in $K^n$ is the ideal $I_q:=I+\left<X_1^q-X_1,\ldots,X_n^q-X_n\right>$. 
\end{tabular}
\end{center}

In line with the terminology of Ghorpade~\cite{G}, we call it the \emph{Affine $\Fq$-Nullstellensatz}. For a comprehensive survey of this result, we refer to the expository article by Ghorpade~\cite{G}. Moreover, this result indicates that the ideal $I_q$ is radical.

The projective Nullstellensatz is an analogue of the Nullstellensatz for projective algebraic varieties and homogeneous ideals of polynomial rings. In the case of an algebraically closed field $K$, it is well-known and can be readily derived as a consequence of the Hilbert Nullstellensatz, particularly by noting that the radical of a homogeneous ideal is also a homogeneous ideal. 
In this article, we remark that when $K$ is infinite, a projective Nullstellensatz can be obtained using the Hilbert $K$-Nullstellensatz. In the case of finite fields, Mercier and Rolland~\cite{MR} proved the projective analogue of the Very Weak Nullstellensatz. 
A simpler proof of this result is presented in Ghorpade~\cite{G}, utilizing the concept of projective reduction which was introduced in Beelen, Datta and Ghorpade~\cite{BDG}. Furthermore,  
using the projective analogue of the Hilbert Nullstellensatz, one can derive the projective Nullstellensatz for finite fields, which states that:
\begin{center}
\begin{tabular}{ m{34em} }
Suppose $K=\Fq$ and let $I\subseteq k[X_0,X_1,\ldots,X_n]$ be a homogeneous ideal. Then the vanishing ideal of the set of common zeros of the polynomials of $I$ in the $n$-dimensional projective space over $K$ is the radical of the homogeneous ideal $I_q^*:=I+\left<X_i^qX_j-X_j^q X_i: 0\le i<j\le n\right>$. 
\end{tabular}
\end{center}

For $k=K=\Fq$, a direct proof of this result can be found in Jaramillo, Vaz Pinto and Villarreal~\cite[Theorem 3.13]{JVV}. Beelen, Datta and Ghorpade in \cite{BDG} showed that, unlike in the affine case, the ideal $I_q^*$ is not always radical by providing an illustrative example. To strengthen the projective Nullstellensatz for $k=K=\Fq$, Gimenez, Ruano and San-Jos\'e~\cite{GRS} proved that the saturation of the ideal $I_q^*$ with respect to the homogeneous maximal ideal $\left<X_0,\ldots,X_n\right>$ of $k[X_0,\ldots,X_n]$ also provides a projective Nullstellensatz, which is computationally more efficient. We improve their result by proving that the ideal quotient of $I_q^*$ with respect to the ideal $\left<X_0^d,\ldots,X_n^d\right>$ in $k[X_0,\ldots,X_n]$ for a fixed positive integer $d$ suffices for the projective Nullstellensatz over finite fields in the general setting $k\subseteq K$. It further leads to a direct and elegant alternative proof of the projective Nullstellensatz for finite fields.

Finding the most effective form of the Nullstellensatz for an arbitrary field is still a subject of research. Although we have the Hilbert $K$-Nullstellensatz, the $K$-radical is an abstract set that is not easy to determine. Laksov and Westin~\cite{LW} conjectured four Nullstellens\"atze to strengthen the Hilbert $K$-Nullstellensatz. We prove that three of these conjectures are not true by providing a common counterexample. 

The content of the paper is organized as follows. In Section 2, we recall several Nullstellens\"atze, including the Hilbert $K$-Nullstellensatz and the Affine $\Fq$-Nullstellensatz. We also provide an outline of Laksov's proof of the Hilbert $K$-Nullstellensatz. In Section~3, we present our results on the projective Nullstellensatz for arbitrary fields. Finally, in Section~4, we discuss the conjectures of Laksov and Westin on Nullstellens\"atze for arbitrary fields.

\section{Preliminaries}

Throughout this paper, let $k$ be a field and let $\bar{k}$ be an algebraic closure of $k$. 
Suppose $K$ is a subfield of $\bar{k}$ containing $k$. Let $q$ be a prime power and let $\Fq$ denote the finite field with $q$ elements. Fix $n$ a nonnegative integer. Denote by $\A^n_K$ the $n$-dimensional affine space over the field $K$, i.e., the set of elements of the vector space $K^{n}$.  
Let $R$ denote the polynomial ring $k[X_1,\ldots,X_n]$ in $n$ variables over $k$. For any subset $A$ of $R$, define the \emph{zero set of $A$ in $\A^n_K$} as
$$
\ZZ_K(A)=\{(a_1,\ldots,a_n)\in \A^n_K: f(a_1,\ldots,a_n)=0 \text{ for all }f\in A\}.
$$ 
Let $I$ be an ideal of $R$ and let $A$ be a set of polynomials which generate $I$. 
Then it is clear that $\ZZ_K(I)=\ZZ_K(A)$. A subset $Z$ of $\A^n_K$ is called an \emph{affine algebraic variety defined over $k$} if there exists a subset $A$ of $R$ such that $Z=\ZZ_K(A)$. We are following the terminology of Cox, Little and O'Shea~\cite{CLO} and what we have called an affine algebraic variety is sometimes referred to as an affine algebraic set. For an affine algebraic variety $Z\subseteq \A^n_K$, the \emph{vanishing ideal of $Z$ in $R$} is defined as 
$$
\I(Z):=\{f\in R: f({\bf a})=0 \text{ for all }{\bf a}\in Z\}.
$$
It is easy to see that $\ZZ_K(\I(Z))=Z$ when $Z$ is an affine algebraic variety. 

\subsection{Hilbert $K$-Nullstellensatz}

Let $y_0,y_1,\ldots$ be a countably infinite set of elements that are algebraically independent over $k$. 
\begin{defn}
For a nonnegative integer $m$, define the set 
$$
P_K(m):=\{p\in k[y_0,\ldots,y_m]: p \text{ is homogeneous and }\ZZ_K(p)\subseteq \ZZ_K(y_0)\}.
$$
\end{defn}

\begin{defn}
Let $I$ be an ideal of $R$. The $K$-radical of $I$ is defined as 
\begin{eqnarray*}
& \sqrt[K]{I}:=& \{f\in R: \text{ there exist an integer }m\ge 0 ,\ p \in P_K(m) \text{ and }f_1,\ldots, f_{m}\in R\\
 & & \hspace{1.5cm} \text{such that }p(f,f_1,\ldots, f_{m})\in I\}.
\end{eqnarray*}
\end{defn}

The next two propositions specify certain properties of the $K$-radical of an ideal $I$ of $R$, the proofs of which can be found in Laksov~\cite[Section~2]{L}. Additionally, for further insight into the $K$-radical of $I$, we include the proof of the second proposition here.

\begin{prop}\label{prop:1}
For an ideal $I$ of $R$, the $K$-radical of $I$ is a radical ideal of $R$, i.e,  if $H:=\sqrt[K]{I}$, then $\sqrt{H}=H$. Moreover,
$$
\sqrt[K]{H}=H.
$$

\end{prop}

\begin{prop}\label{prop:2}
Let $I$ be an ideal of $R$. Write $H:=\sqrt[K]{I}$. Then $\sqrt{I}\subseteq H$ and 
$$
\ZZ_K(I)=\ZZ_K(H).
$$
\end{prop}

\begin{proof}
Note that $I\subseteq H$. Indeed, if $f\in I$, then the polynomial $p(y_0):=y_0$ in $P_K(0)$ is such that $p(f)=f\in I$, which implies that $f\in H$. Thus by Proposition \ref{prop:1}, $\sqrt{I}\subseteq H$. For the second part, it is clear that $\ZZ_K(H)\subseteq \ZZ_K(I)$ because $I\subseteq H$. 
To prove the other containment, let ${\bf a}\in \ZZ_K(I)$. For $f\in H$, there exist a nonnegative integer $m$, a polynomial $p\in P_K(m)$ and the polynomials $f_1,\ldots,f_m\in R$ such that $p(f,f_1,\ldots,f_m)\in I$. Thus, $p(f({\bf a}),f_1({\bf a}),\ldots,f_m({\bf a}))=0$. Since $p\in P_K(m)$, we obtain $f({\bf a})=0$, as required. 
\end{proof}

Now we are ready to state a Nullstellensatz for an arbitrary pair of fields $k\subseteq K$, known by Adkins, Gianni and Tognoli~\cite{AGT}, and Laksov~\cite{L}. 
Following Laksov~\cite{L}, we call it the \emph{Hilbert $K$-Nullstellensatz}. 

\begin{theorem}[Hilbert $K$-Nullstellensatz]\label{thm:null}
Let $I$ be an ideal of $R$. Then
$$
\I(\ZZ_K(I))=\sqrt[K]{I}.
$$
\end{theorem}

\begin{remark}\label{Rmk:NullAlgClosed}
{\rm The Hilbert Nullstellensatz follows from the Hilbert $K$-Nullstellensatz because if $K=\bar{k}$, then $\sqrt[K]{I}=\sqrt{I}$. Indeed, for any nonnegative integer $m$ and for any $p\in P_K(m)$, the polynomial  $p(a,y_1,\ldots,y_m)$ has no zeros in $\A^{m}_K$ for $a\in K\setminus \{0\}$. Since $K$ is algebraically closed, we derive that $p(a,y_1,\ldots,y_m)$ is a nonzero constant polynomial. But $p(y_0,y_1,\ldots,y_m)$ is homogeneous, so $p=ay_0^s$ for some integer $s\ge 0$ and $a\in k$. Therefore, $P_K(m)=\{a y^s: a\in k \text{ and } s\ge 0\}$ for any nonnegative integer $m$. }
\end{remark}

In view of the above remark, we could state the Hilbert Nullstellensatz as a corollary of the Hilbert $K$-Nullstellensatz.  

\begin{cor}[Hilbert Nullstellensatz]
Suppose $K=\bar{k}$ and let $I$ be an ideal of $R$. Then 
$$
\I(\ZZ_K(I))=\sqrt{I}.
$$
\end{cor}

We outline the proof of the Hilbert $K$-Nullstellensatz following Laksov's approach~\cite{L}.  
As mentioned in the introduction, the proof is by Rabinowitsch trick~\cite{R} assuming the weak form of the Hilbert $K$-Nullstellensatz. The statement of the weak form of the Hilbert $K$-Nullstellensatz can be obtained by substituting $\ZZ_K(I)=\emptyset$ in the Hilbert $K$-Nullstellensatz. As in Laksov~\cite{L}, we refer to it as the \emph{Weak Hilbert $K$-Nullstellensatz}, which is as follows.

\begin{theorem}[Weak Hilbert $K$-Nullstellensatz]\label{thm:weakNull}
Let $I$ be an ideal of $R$. If $\ZZ_K(I)$ is empty, then $1\in \sqrt[K]{I}$. 
\end{theorem}

By Remark \ref{Rmk:NullAlgClosed}, $\sqrt[K]{I}=\sqrt{I}$ when $K=\bar{k}$. Thus, in this case, the Weak Hilbert $K$-Nullstellensatz follows from the weak form of the Hilbert Nullstellensatz and there are various proofs available in the literature, for instance, see Atiyah and MacDonald~\cite[Corollary 7.10]{AM}. 
Now suppose $K$ is not algebraically closed. Then the Weak Hilbert $K$-Nullstellensatz uses the following standard result, a proof of which can be found in Laksov~\cite[Proposition 5]{L}. 

\begin{prop}\label{prop:trivialzero}
Suppose $K$ is not algebraically closed. Then for each positive integer $e$, there exists a homogeneous polynomial $p\in k[y_1,\ldots,y_e]$ such that $\ZZ_K(p)=\{\bf{0}\}$ where $\bf{0}$ represents the zero vector in $\A^e_K$.
\end{prop}

\noindent\emph{Proof of Theorem \ref{thm:weakNull} when $K$ is not algebraically closed.} Since $R$ is a Noetherian ring, there exist finitely many elements $h_1,\ldots,h_r\in R$ such that $I=\left<h_1,\ldots,h_r\right>$. Note that $\ZZ_K(I)=\emptyset$ implies that $I$ is not the zero ideal, and hence $r\ge 1$. By Proposition~\ref{prop:trivialzero}, there exists a homogeneous polynomial $p\in k[y_1,\ldots,y_r]$ such that $p$ has only the trivial zero in $\A^r_K$. Then the homogeneous degree of $p$ is $\ge 1$. Consider the polynomial $g$ in $R$ defined by
$$
g(X_1,\ldots,X_n):=p(h_1,\ldots,h_r).
$$
Since $\ZZ_K(I)$ is empty, $h_1,\ldots,h_r$ have no common zeros and thus $g(X_1,\ldots,X_n)$ has no zeros in $\A^n_K$. Now homogenize $g$ using the new variable $X_0$, call it $g^h(X_0,\ldots,X_n)$. Then 
$$
g^h(1,X_1,\ldots,X_n)=g(X_1,\ldots,X_n)=p(h_1,\ldots,h_r),
$$
which is in $I$ 
because $p$ is homogeneous of degree $\ge 1$ and 
$h_1,\ldots,h_r\in I$. Thus $g^h\in P_K(n)$ and $1\in \sqrt[K]{I}$. 
$\hfill\square$

\vspace{3mm}

\noindent\emph{Proof of Theorem \ref{thm:null}.} By Proposition \ref{prop:2}, it is clear that $\sqrt[K]{I}\subseteq \I(\ZZ_K(I))$. Let $f\in \I(\ZZ_K(I))$. Let $I$ be generated by the polynomials $h_1,\ldots,h_r$ in $R$. We now use the Rabinowitsch trick~\cite{R}. Consider the polynomial ring $R[X_{n+1}]=k[X_1,\ldots,X_{n+1}]$ in $n+1$ variables. Let $J$ be an ideal of $R[X_{n+1}]$ generated by the polynomials
$$
h_1,\ldots,h_r,1-X_{n+1}f.
$$
Then $J$ has no common zeros in $\A^{n+1}_K$. Thus, by the Weak Hilbert $K$-Nullstellensatz, $1\in\sqrt[K]{J}$. This implies that there exist nonnegative integer $m$, polynomial $p\in P_K(m)$ and polynomials $f_1,\ldots,f_m\in R[X_{n+1}]$ such that $p(1,f_1,\ldots,f_m)\in J$. Note that, since $p\in P_K(m)$, the polynomial $p(1,f_1,\ldots,f_m)$ has no zeros in $\A^{n+1}_K$. Let
$$
g(X_1,\ldots,X_{n+1}):=p(1,f_1,\ldots,f_m).
$$
Then there exist $g_1,\ldots,g_{r+1}\in R[X_{n+1}]$ such that
$$
g(X_1,\ldots,X_{n+1})=\sum_{i=1}^{r} g_i h_i+g_{r+1} (1-X_{n+1} f).
$$
Homogenize $g$ using the new variable $X_0$ and call it $g^h(X_0,X_1,\ldots,X_{n+1})$. Then, 
$$
g^h(1,X_1,\ldots,X_{n+1})=g(X_1,\ldots,X_{n+1})\in J.
$$
Since $g$ has no zeros in $\A^{n+1}_K$, we acquire $g^h\in P_K({n+1})$. Now we claim that the polynomial $g^h(f,fX_1,\ldots,fX_{n},1)$ is a linear combination of $h_1,\ldots,h_r$. First note that if the polynomial $\ell^h(X_0,X_1,\ldots,X_n)$ denotes the homogenization of a polynomial $\ell\in R$ using the variable $X_0$, then $\ell^h(f,fX_1,\ldots,fX_{n})=f^{\deg \ell} \ell(X_1,\ldots,X_n)$. Thus, the homogeneous polynomial $h_i^h(f,fX_1,\ldots,fX_{n})$ is an $R$-multiple of $h_i$ for each $1\le i\le r$. Moreover, the homogenization of $(1-X_{n+1} f)$ using the variable $X_0$, say $(1-X_{n+1}f)^h$, is $X_0^{\deg f+1}-X_{n+1}f^h(X_0,X_1,\ldots,X_n)$. Then,
$$
(1-X_{n+1}f)^h(f,fX_1,\ldots,fX_{n},1)=(f^{\deg f+1}-1\cdot f^{\deg f} f)=0.
$$
This proves the claim, i.e., $g^h(f,fX_1,\ldots,fX_{n},1)\in I$. 
Hence, $f\in \sqrt[K]{I}$.  
$\hfill\square$

\vspace{3mm}
The following statement by Gallego, V\'elez and G\'omez-Ram\'rez \cite{GVG} for the weak form of the Nullstellensatz over an arbitrary field could be of interest, which only involves the local existence of zeros of polynomials of the ideal. They call it the \emph{Bezout's form of the Nullstellensatz for arbitrary fields}. We can see this as a corollary of the Weak Hilbert $K$-Nullstellensatz. 

\begin{cor}[Bezout's form of the Nullstellensatz for arbitrary fields]\label{cor:BezoutNull}
Let $I$ be an ideal of $R$. For any field~$K$, $\ZZ_K(I)=\emptyset$ if and only if there exists a polynomial in $I$ which has no zeros in $\A^n_K$.
\end{cor}
\begin{proof}
One direction is clear. For the other direction, let $\ZZ_K(I)=\emptyset$. Then by Theorem~\ref{thm:weakNull}, $1\in \sqrt[K]{I}$, i.e., there exist nonnegative integer $m$, homogeneous polynomial $p\in P_K(m)$ and $f_1,\ldots,f_{m}\in R$ such that $p(1,f_1,\ldots,f_m)\in I$. Since $\ZZ_K(p)\subseteq \ZZ_K(y_0)$ in $\A^{m+1}_K$, the polynomial $p(1,f_1,\ldots,f_m)$ has no zeros in $\A^n_K$. 
\end{proof}

To keep the article self-contained, we also state the Very Weak Nullstellensatz as mentioned in the introduction. 
This result holds for any infinite field and can be easily proved by induction on $n$. 

\begin{theorem}[Very Weak Nullstellensatz]
Assume that $K$ is an infinite field. If $f\in R$ vanishes at every point of $\A^n_K$, then $f$ is the zero polynomial. 
\end{theorem}

By the Hilbert $K$-Nullstellensatz, it can be observed that the Very Weak Nullstellensatz implies that the $K$-radical of the ideal $\left<0\right>$ is  $\left<0\right>$ when $K$ is infinite. 

\subsection{Affine $\Fq$-Nullstellensatz}
In this subsection, we assume that $K=\Fq$. As earlier, $k$ denotes a subfield of $K$. In this case, an explicit form of the Hilbert $K$-Nullstellensatz is given by Tarjanian~\cite{T}. 
For a historical account, interested readers can refer to Ghorpade~\cite{G}. For any field $F$, let $\Gamma_q(F)$ denote the ideal of $F[X_1,\ldots,X_n]$ generated by the polynomials $X_1^q-X_1,\ldots,X_n^q-X_n$. Here, we refer to the corresponding results of the Nullstellensatz from the previous subsection as the \emph{Very Weak $\Fq$-Nullstellensatz, Weak $\Fq$-Nullstellensatz}, and the \emph{Affine $\Fq$-Nullstellensatz}. 
We observe that the proofs referred to below for the Affine $\Fq$-Nullstellensatz make strong use of the Very Weak $\Fq$-Nullstellensatz, unlike the proofs of the Hilbert Nullstellensatz and the Hilbert $K$-Nullstellensatz where their weak forms used more prominently. 
We present the results in a general setting here and refer to the proofs that are more relevant to our work on the projective Nullstellensatz in the following section. 

\begin{theorem}[Very Weak $\Fq$-Nullstellensatz]
The vanishing ideal of the affine space $\A^n_K$ in $R$ is 
\begin{equation}\label{eq:FiniteNull1}
\I(\A^n_K)=\Gamma_q(k).
\end{equation}
Moreover, if $F$ is an algebraic extension of $K$, then the vanishing ideal $\I_F(\A^n_K)$ considered as an ideal in $F[X_1,\ldots,X_n]$ satisfies  
\begin{equation}\label{eq:FiniteNull2}
\I_F(\A^n_K)=\Gamma_q(F).
\end{equation}
\end{theorem}
\begin{proof}
The result \eqref{eq:FiniteNull2} follows from Ghorpade~\cite[Theorem 2.3(i)]{G}. Thus for $F=K=\Fq$, we obtain $\I_K(\A^n_K)=\Gamma_q(K)$. Now,  
$$
\I(\A^n_K)=\I_K(\A^n_K)|_{k[X_1,\ldots,X_n]}=\Gamma_q(K)|_{k[X_1,\ldots,X_n]}=\Gamma_q(k),
$$
where the last equality follows from the fact that the generators $X_1^q-X_1,\ldots,X_n^q-X_n$ of $\Gamma_q(K)$ have coefficients in $k$. 
\end{proof}

\begin{theorem}[Affine $\Fq$-Nullstellensatz]\label{thm:finiteNull}
Let $I$ be an ideal of $R$. 
Then
\begin{equation}\label{eq:FiniteNull3}
\I(\ZZ_{K}(I))=I+\Gamma_q(k).
\end{equation}
Moreover, if $F$ is an algebraic extension of $K$ and $J$ is an ideal of $F[X_1,\ldots,X_n]$ generated by some polynomials of $K[X_1,\ldots,X_n]$, then the vanishing ideal $\I_F(\ZZ_K(J))$ considered as an ideal in $F[X_1,\ldots,X_n]$ satisfies  
\begin{equation}\label{eq:FiniteNull4}
\I_{F}(\ZZ_{K}(J))=J+\Gamma_q(F).
\end{equation}
\end{theorem}
\begin{proof}
Note that by Ghorpade~\cite[Theorem 2.3]{G}, the result \eqref{eq:FiniteNull4} is true. To prove \eqref{eq:FiniteNull3}, let $I$ be generated by $h_1,\ldots,h_r$ in $R$. Since $k\subseteq K=\Fq$, consider the ideal $J$ of $K[X_1,\ldots,X_n]$ generated by $h_1,\ldots,h_r$. Thus for $F=K=\Fq$, by \eqref{eq:FiniteNull4}, we know that 
\begin{equation}\label{eq:3}
\I_{K}(\ZZ_{K}(J))=J+\Gamma_q(K).
\end{equation}
Since $\ZZ_{K}(J)=\ZZ_{K}(h_1,\ldots,h_r)=\ZZ_{K}(I)$, we obtain
\begin{equation}\label{eq:4}
\I_{K}(\ZZ_{K}(J))|_{k[X_1,\ldots,X_n]}=\I(\ZZ_{K}(I)).
\end{equation}
Combining \eqref{eq:3} and \eqref{eq:4}, we get
$$
\I(\ZZ_{K}(I))=(J+\Gamma_q(K))|_{k[X_1,\ldots,X_n]}. 
$$
Since $J+\Gamma_q(K)$ as an ideal of $K[X_1,\ldots,X_n]$, is generated by polynomials $h_1,\ldots,h_r,$ $X_1^q-X_1,\ldots,X_n^q-X_n$ which are also in $k[X_1,\ldots,X_n]$, we derive that
$$
(J+\Gamma_q(K))|_{k[X_1,\ldots,X_n]}=I+\Gamma_q(k),
$$
as required.
\end{proof}

A notable consequence of the above theorem is that 
$I+\Gamma_q(k)$ is a radical ideal of $R$. We now state the Weak $\Fq$-Nullstellensatz which can be deduced by the Affine $\Fq$-Nullstellensatz when $\ZZ_K(I)=\emptyset$. 

\begin{cor}[Weak $\Fq$-Nullstellensatz]
Let $I$ is an ideal of $R$. If $\ZZ_K(I)$ is empty, then 
$$
1+f\in I \quad \text{ for some } f\in \Gamma_q(k).
$$
Moreover, if $F$ is an algebraic extension of $K$ and $J$ is an ideal of $F[X_1,\ldots,X_n]$ generated by some polynomials of $K[X_1,\ldots,X_n]$ such that $\ZZ_K(J)$ is empty, then 
\begin{equation*}
1+f\in J \quad \text{ for some } f\in \Gamma_q(F).
\end{equation*}
\end{cor}

\section{Main result: Projective Nullstellensatz}
Denote by $\PP^n_K$ the $n$-dimensional projective space over the field $K$ which is defined as the set $({K^{n+1}\setminus \{\bf{0}\}})/{\sim}$ of equivalence classes of ${K^{n+1}\setminus \{\bf{0}\}}$ under the proportionality relation, i.e., for all $u,v\in K^{n+1}\setminus \{\bf{0}\}$, 
$$
u\sim v \Longleftrightarrow u=\lambda v \text{ for some }\lambda\in K.
$$
We denote the equivalence classes of $(a_0,\ldots,a_{n})$ in $K^{n+1}\setminus \{\bf{0}\}$ by $[a_0:\cdots:a_n]$.  
Let $S$ be the polynomial ring $k[X_0,\ldots,X_n]$ in $n+1$ variables over the field $k$. 
For any subset $B$ of $S$ consisting of homogeneous polynomials, define the \emph{zero set of $B$ in $\PP^n_K$} as
$$
\V_K(B):=\{[a_0:\cdots:a_n]\in \PP^n_K:f(a_0,\ldots,a_n)=0 \text{ for all }f\in B\}.
$$
Let $I$ be a homogeneous ideal of $S$ and let $B$ be a set of homogeneous polynomials which generates $I$. Then the \emph{zero set of $I$ in $\PP^n_K$} is defined as $\V_K(I):=\V_K(B)$. A subset $V$ of $\PP^n_K$ is called a \emph{projective (algebraic) variety defined over $k$} if there exists a subset $B$ of $S$ of homogeneous polynomials such that $V=\V_K(B)$. For a projective variety $V\subseteq \PP^n_K$, the \emph{vanishing ideal of $V$ in $S$} is defined as 
$$
\I(V):=\left<f\in S: f \text{ is homogeneous and } f({\bf a})=0 \text{ for every }{\bf a}\in V\right>.
$$
It is easy to show that $\V_K(\I(V))=V$ if $V$ is a projective variety. 

\subsection{Projective Nullstellensatz for infinite fields}
The projective Nullstellensatz is a projective analogue of the Nullstellensatz for the projective variety $V$ in $\PP^n_K$ and the homogeneous ideal $\I(V)$ in $S$. When $K$ is infinite, we show that the projective Nullstellensatz can be directly obtained using the Nullstellensatz of the affine case. For this, we use a simple observation noted, for instance, in Fulton~\cite[Problem~4.2]{F}.

\begin{prop}\label{affprojsame}
Assume that $K$ is infinite. Let $f\in S$. Write $f:=f_1+\cdots+f_s$ where $f_1,\ldots,f_s$ are the homogeneous components of $f$ of degrees $d_1,\ldots,d_s$ respectively satisfying $d_1>\cdots>d_s\ge 0$. If there exists ${\bf a}=[a_0:\cdots:a_n]\in \PP^n_K$ such that $f(\lambda a_0,\ldots,\lambda a_n)=0$ for every $\lambda\in K\setminus \{0\}$, then $f_i(\lambda a_0,\ldots,\lambda a_n)=0$ for every $1\le i\le s$ and for every $\lambda\in K\setminus \{0\}$, or simply $f_i(a_0,\ldots, a_n)=0$ for every $1\le i\le s$.     
\end{prop}
\begin{proof}
Note that 
$$
f(\lambda a_0,\ldots,\lambda a_n)=\sum_{i=1}^s \lambda^{d_i} f_i(a_0,\ldots,a_n)=0 \quad \text{for }\lambda\in K\setminus \{0\}.
$$
Write $b_i:=f_i(a_0,\ldots,a_n)$ for $1\le i\le s$. Then $b_i\in K$ for all $1\le i\le s$. Define 
$$
g(x):=b_1x^{d_1}+\cdots+b_s x^{d_s}
$$
in the polynomial ring $K[x]$ with variable $x$. Then $g(\lambda)=0$ for all $\lambda\in K\setminus \{0\}$. Since $K$ is infinite, $g$ has infinitely many zeros in the field $K$ which is possible only if $g=0$. Thus, $b_1=\cdots=b_s=0$ because $d_1,\ldots,d_s$ are distinct and nonnegative. This  proves the result.
\end{proof}

\begin{cor}\label{cor:projinfsame}
Let $I$ be a homogeneous ideal of $S$. If $K$ is infinite and $\V_K(I)\ne \emptyset$, then 
$$
\I(\V_K(I))=\I(\ZZ_K(I)).
$$
\end{cor}

\begin{proof}
Since $\V_K(I)\ne \emptyset$, the homogeneous ideal $I$ can not contain a nonzero constant polynomial.  
This gives ${\bf 0}\in \ZZ_K(I)\subseteq \A^{n+1}_K$ and $\V_K(I)=(\ZZ_K(I)\setminus \{{\bf 0}\})/\sim$. Thus, $\I(\V_K(I))$ is the largest homogeneous ideal contained in $\I(\ZZ_K(I))$. But, by Proposition~\ref{affprojsame}, $\I(\ZZ_K(I))$ is itself a homogeneous ideal. Hence, $\I(\V_K(I))=\I(\ZZ_K(I))$. 
\end{proof}

By the above corollary, the Hilbert $K$-Nullstellensatz can be used to obtain the projective Nullstellensatz when $K$ is infinite and $\V_K(I)\ne \emptyset$. We complete the result by proving the following theorem.

\begin{theorem}[Projective Nullstellensatz for infinite fields]\label{thm:projNullinf}
Assume that $K$ is infinite. Then for any homogeneous ideal $I$ of $S$, 
\begin{enumerate}
\item $\I(\V_K(I))=\sqrt[K]{I}$ if $\V_K(I)$ is nonempty.
\item $\I(\V_K(I))=S$ if $\V_K(I)$ is empty. 
In this case, 
either $I=S$ or $\sqrt[K]{I}=\left<X_0,\ldots,X_n\right>$.
\end{enumerate}
Additionally, in case $K$ is algebraically closed, 
\begin{enumerate}
\item[(3)] $\I(\V_K(I))=\sqrt{I}$ if $\V_K(I)$ is nonempty.
\item[(4)] $\I(\V_K(I))=S$ if $\V_K(I)$ is empty. In this case, either $I=S$ or $\sqrt{I}=\left<X_0,\ldots,X_n\right>$.
\end{enumerate}
\end{theorem}
\begin{proof}
(1) This follows from Corollary \ref{cor:projinfsame} and Theorem \ref{thm:null}.

(2) Clearly, $1\in \I(\V_K(I))$ when $\V_K(I)$ is empty. This implies that $\I(\V_K(I))=S$. Since $\V_K(I)$ is empty and $I$ is a homogeneous ideal, we have either $\ZZ_K(I)=\emptyset$ or $\ZZ_K(I)=\{{\bf 0}\}$ in $\A^{n+1}_K$. In the first case, by Corollary \ref{cor:BezoutNull}, there exists a polynomial $f$ in $I$ such that $\ZZ_K(f)=\emptyset$. Since $I$ is homogeneous, all the homogeneous components of $f$ are in $I$. Thus, ${\bf 0}\notin \ZZ_K(f)$ implies that there exists a constant component of $f$ and hence $1\in I$. In the second case, i.e., when $\ZZ_K(I)=\{{\bf 0}\}$, we obtain $\sqrt[K]{I}=\I(\{{\bf 0}\})=\left<X_0,\ldots,X_n\right>$ by Theorem~\ref{thm:null}.

In the case of $K=\bar{k}$, (3) and (4) follow from the fact that $\sqrt[K]{I}=\sqrt{I}$, as observed in Remark~\ref{Rmk:NullAlgClosed}. 
\end{proof}

\begin{remark}\label{rem:finiteCounter}
\rm{The $K$-radical of an ideal of $S$ need not be homogeneous when $K$ is finite. For example, let $K=k=\F_2$. Consider the ring $S=\F_2[X_0,X_1]$ and the ideal $I=\left<X_0\right>$ in $S$. Then $I+\Gamma_q(\F_2)=\left<X_0,X_1^2-X_1\right>$. Thus, by the Affine $\Fq$-Nullstellensatz and the Hilbert $K$-Nullstellensatz, $\sqrt[\F_2]{I}=\left<X_0,X_1^2-X_1\right>$, which is not a homogeneous ideal because $X_1\notin \I(\ZZ_{\F_2}(I))$. Therefore, for the finite field $K$, the vanishing ideal of a nonempty projective variety $\V_K(I)$ for a homogeneous ideal $I$ of $S$ can be a proper subset of the $K$-radical of the ideal $I$. 
}
\end{remark}

\subsection{Projective Nullstellensatz for finite fields}
For any field $F$, let $\Gamma_q^*(F)$ denote the ideal in $F[X_0,\ldots,X_n]$ generated by all polynomials of the form $X_i^q X_j-X_j^q X_i$ for $0\le i<j\le n$. Assume that $K=\Fq$, and recall that $k\subseteq K\subseteq \bar{k}$. Mercier and Rolland ~\cite{MR} proved the very weak form of the projective Nullstellensatz for finite fields by determining the vanishing ideal of the projective space $\PP^n_{K}$ in the polynomial ring $S=k[X_0,\ldots,X_n]$. We refer to this result as the \emph{Very Weak Projective $\Fq$-Nullstellensatz}. Below, we present a more general version of this result and refer to the most relevant proofs here. 

\begin{theorem}[Very Weak Projective $\Fq$-Nullstellensatz]\label{thm:VeryWeakFiniteProj}
The vanishing ideal of $\PP^n_K$ in $S$ is 
\begin{equation}\label{eq:ProjFiniteNull1}
\I(\PP^n_K)=\Gamma_q^*(k).
\end{equation}
Moreover, if $F$ is an algebraic extension of $K$, then the vanishing ideal $\I_F(\PP^n_K)$ considered as an ideal of $F[X_0,\ldots,X_n]$ satisfies  
\begin{equation}\label{eq:ProjFiniteNull2}
\I_F(\PP^n_K)=\Gamma_q^*(F).
\end{equation}
\end{theorem}
\begin{proof}
The proof of \eqref{eq:ProjFiniteNull2} follows from Ghorpade~\cite[Theorem 3.3]{G} (or Beelen, Datta and Ghorpade~\cite[Corollary 2.6]{BDG}). Thus for $F=K=\Fq$, we obtain $\I_K(\PP^n_K)=\Gamma_q^*(K)$. Now,  
$$
\I(\PP^n_K)=\I_K(\PP^n_K)|_{k[X_0,\ldots,X_n]}=\Gamma_q^*(K)|_{k[X_0,\ldots,X_n]}=\Gamma_q^*(k), 
$$
where the last equality follows because of the fact that the generators $X_i^q X_j-X_j^q X_i$, $0\le i<j\le n$ of $\Gamma_q^*(K)$ are also in $k[X_0,\ldots,X_n]$. 
\end{proof}

As mentioned in Gimenez, Ruano and San-Jos\' e~\cite{GRS}, the projective Nullstellensatz can be obtained as a corollary of the Hilbert Nullstellensatz. To see this,  let $I$ be a homogeneous ideal of $S$. If $\V_K(I)\ne \emptyset$, then
\begin{eqnarray*}
	\I(\V_{K}(I))=\I(\V_{\overline{k}}(I+\Gamma_q^*(k)))=\sqrt{I+\smash[b]{\Gamma_q^*}{(k)}}
\end{eqnarray*}
where the first equality follows from the fact that $\PP^n_{K}=\V_{\overline{k}}(\Gamma_q^*(k))$. Note that this kind of argument does not work if $K$ is an infinite, non-algebraically closed field, because in that case, $\PP^n_K$ is not a projective variety in the projective space $\PP^n_{\overline{k}}$. For $k=K=\Fq$, a direct proof of this projective Nullstellensatz for finite fields, without using the Hilbert Nullstellensatz, is given in Jaramillo, Vaz Pinto and Villarreal~\cite[Theorem~3.13]{JVV}. Beelen, Datta and Ghorpade~\cite[Example~3.6]{BDG} (or Ghorpade~\cite[Example~3.5]{G}) observed that $I+\Gamma_q^*(k)$ need not be a radical ideal even if $I$ is a radical ideal, unlike in the case of affine algebraic varieties. In light of this and to strengthen the projective Nullstellensatz, Gimenez, Ruano and San-Jos\' e~\cite{GRS}, inspired by the proof of Jaramillo, Vaz Pinto and Villarreal~\cite{JVV} for the above projective Nullstellensatz, proved that one can obtain the projective Nullstellensatz using a computationally efficient set than the radical of certain ideal in the case of $k=K$. To state their result, we recall some terminology from commutative algebra. Let $\m$ denote the unique homogeneous maximal ideal of $S$, i.e.,  $\m=\left<X_0,\ldots,X_n\right>$. Let $I$ and $J$ be two ideals of $S$. The \emph{ideal quotient of $I$ with respect to $J$} is the set
$$
I:J=\{f\in S: fJ\subseteq I\},
$$
and the \emph{saturation of $I$ with respect to $J$} is the set 
$$
I:J^{\infty} =\bigcup_{j\ge 1} (I:J).
$$
It can be shown that $I:J$ and $I:J^\infty$ are the ideals of $S$ (see, for instance, Cox, Little and O'Shea~\cite[Chapter~4, Section~4]{CLO}). Now, Gimenez, Ruano and San-Jos\' e~\cite{GRS} proved the following: 
Suppose $k=K$ and $I$ is a homogeneous ideal of $S$. If 
 $\V_K(I)\ne \emptyset$, then 
 $$
 \I(\V_{K}(I))=(I+\Gamma^*_q(k)):\m^{\infty}.
 $$

In simple terms, it states that for $k=K$ and for a homogeneous ideal $I$ of $S$, if $\V_K(I)$ is nonempty, then the vanishing ideal of $\V_K(I)$ in $S$ consists of the polynomials $f$ in $S$ for which there exists an integer $j\ge 1$ such that $fX_0^{a_0}\cdots X_n^{a_n}\in I+\Gamma^*_q(k)$ for all tuples $(a_0,\ldots,a_n)\in \Z_{\ge 0}$ satisfying $a_0+\cdots+a_n=j$. In the following theorem, we prove that an even simpler set suffices for the projective Nullstellensatz to hold. This approach provides an alternative, direct proof of the result without relying on the Hilbert Nullstellensatz. In line with the terminologies introduced earlier, we refer to our version of the projective Nullstellensatz as the \emph{Projective $\Fq$-Nullstellensatz}. Note that as before, for any field $F$, the ideal $\Gamma_q(F)$ in $F[X_0,\ldots,X_n]$ denotes the ideal generated by $X_0^q-X_0,\ldots,X_n^q-X_n$.

\begin{theorem}[Projective $\Fq$-Nullstellensatz]\label{thm:projFiniteNull1} 
Let $I$ be a homogeneous ideal of $S$ generated by the homogeneous polynomials $h_1,\ldots,h_r$ of degrees $d_1,\ldots,d_r$, respectively. Define $d:=~\hspace{-2mm}(d_1+\cdots+d_r)(q-1)+1$. 
Let $\mathfrak{d}$ denote the ideal generated by the $d$-th powers of $X_0,\ldots,X_n$ in $S$, i.e., $\mathfrak{d}=\left<X_0^d,\ldots,X_n^d\right>$. Then    
\begin{enumerate}
\item $\I(\V_{K}(I))=(I+\smash[b]{\Gamma_q^*}{(k)}):\mathfrak{d}$ if $\V_K(I)$ is nonempty.
\item $\I(\V_K(I))=S$ if $\V_K(I)$ is empty. In this case, 
 either $I=S$ or $I+\Gamma_q(k)=\m
 $.
\end{enumerate}
Furthermore, let $F$ be an algebraic extension of $K$ and let $J$ be a homogeneous ideal of $F[X_0,\ldots,X_n]$ generated by the homogeneous polynomials $t_1,\ldots,t_u$ of $K[X_0,\ldots,X_n]$ of degrees $s_1,\ldots,s_u$, respectively. If $s:=(s_1+\cdots+s_u)(q-1)+1$, then the vanishing ideal $\I_F(\V_K(J))$ considered as an ideal in $F[X_0,\ldots,X_n]$ satisfies  
\begin{enumerate}
\item[(3)] $\I_F(\V_{K}(J))=(J+\smash[b]{\Gamma_q^*}{(F)}):\mathfrak{s}_F$ if $\V_K(J)$ is nonempty. Here, $\mathfrak{s}_F$ denotes the ideal generated by $X_0^s,\ldots,X_n^s$ in $F[X_0,\ldots,X_n]$. 
\item[(4)] $\I_F(\V_K(J))=F[X_0,\ldots,X_n]$ if $\V_K(J)$ is empty. In this case, 
 either the ideal $J=F[X_0,\ldots,X_n]$ or $J+\Gamma_q(F)=\m_F
 $ where $\m_F$ denotes the ideal generated by $X_0,\ldots,X_n$ in $F[X_0,\ldots,X_n]$. 
\end{enumerate}
\end{theorem}
\begin{proof}
(1) Since $I$, $\Gamma_q^*(k)$ and $\mathfrak{d}$ are homogeneous ideals, $({I+\smash[b]{\Gamma_q^*}(k)}):\mathfrak{d}$ is a homogeneous ideal. Let $f\in ({I+\smash[b]{\Gamma_q^*}(k)}):\mathfrak{d}$ be a homogeneous polynomial. Without loss of generality, we can assume that $f$ is nonzero. Now, observe that $f$ is a non-constant polynomial. Indeed, if $f$ is a nonzero constant polynomial, then $\mathfrak{d}\subseteq I+\Gamma_q^*(k)$, which implies that $\V_K(I)=\V_K(I+\Gamma_q^*(k))\subseteq V_K(\mathfrak{d})=\emptyset$, a contradiction to the hypothesis. Given the polynomial $f\in ({I+\smash[b]{\Gamma_q^*}(k)}):\mathfrak{d}$, it follows that
$$
X_j^{d} f\in I+\Gamma_q^*(k) \quad \text{for each }0\le j\le n.
$$
Since $d\ge 1$, we obtain 
$$
(X_j f)^{d}=X_j^{d} f\cdot f^{d-1}\in I+\Gamma_q^*(k)
$$ 
for $0\le j\le n$. Therefore, 
$X_j f\in \sqrt{I+\smash[b]{\Gamma_q^*}(k)}$ for $0\le j\le n$. In particular, $\mu f\in \sqrt{I+\smash[b]{\Gamma_q^*}(k)}$ for each nonzero term $\mu$ of $f$ because $f$ is a non-constant homogeneous polynomial. Thus, 
$$
f^2=\left(\sum_{\substack{\mu \text{ nonzero}\\ \text{ term of }f}} \mu f \right)\in \sqrt{I+\smash[b]{\Gamma_q^*}(k)}.
$$
Consequently, $f\in \sqrt{I+\smash[b]{\Gamma_q^*}(k)}$. Thus, 
$$(I+\Gamma_q^*(k)):\mathfrak{d}\subseteq \I(\V_K((I+\Gamma_q^*(k)):\mathfrak{d}))\subseteq \I(\V_K(\sqrt{I+\smash[b]{\Gamma_q^*}{(k)}}\hspace{0.4mm}))=\I(\V_K(I)).$$
For the other containment, let $f\in \mathsf{I}(\V_K(I))$ be a homogeneous polynomial. 
Let $0\le j\le n$. Define 
$$
g_j(X_0,\ldots,X_n):=X_j^{d}-X_j\prod_{i=1}^r \left(X_j^{d_i(q-1)}-h_i^{q-1}\right)
$$
where $h_1,\ldots,h_r$ are homogeneous polynomials of degrees $d_1,\ldots,d_r$, respectively, that generate $I$, and $d=(d_1+\cdots+d_r)(q-1)+1$, as defined in the theorem. Then $g_j$ is a homogeneous polynomial in $I$ of degree $d$. Further, consider the homogeneous polynomial
$$
\ell_j(X_0,\ldots,X_n):=X_j^{d}-g_j(X_0,\ldots,X_n) 
$$
in $S$. Then $g_j({\bf a})=0$ for every ${\bf a}\in \V_K(I)$; and $\ell_j({\bf a})=0$ for every ${\bf a}\notin \V_K(I)$. Now 
$$
g_j+\ell_j=X_j^{d}, 
$$
which implies that 
$$
g_jf+\ell_jf=X_j^{d} f.
$$
Here, $g_jf\in I$ because $g_j\in I$. Moreover, $\ell_jf$ vanishes on all the points of $\PP^n_K$ and is homogeneous because $\ell_j$ and $f$ are homogeneous. Thus by Theorem \ref{thm:VeryWeakFiniteProj}, $\ell_jf\in \Gamma_q^*(k)$ and hence 
$$
X_j^{d} f\in I+\Gamma_q^*(k).
$$
Since this is true for any $0\le j\le n$, we obtain  $f\in (I+\smash[b]{\Gamma_q^*}{(k)}):\mathfrak{d}$. This concludes the proof. 

(2) We follow similar steps as we did in the proof of part (2) of Theorem \ref{thm:projNullinf}. It is clear that $1\in \I(\V_K(I))$ when $\V_K(I)$ is empty. Thus, $\I(\V_K(I))=S$. Since $I$ is a homogeneous ideal, we have either $\ZZ_K(I)=\emptyset$ or $\ZZ_K(I)=\{{\bf 0}\}$ in $\A^{n+1}_K$. In the first case, by Corollary~\ref{cor:BezoutNull}, there exists a polynomial $f$ in $I$ such that $\ZZ_K(f)=\emptyset$. Since $I$ is homogeneous, all the homogeneous components of $f$ are in $I$. This implies that there exists a nonzero constant component of $f$ and hence $1\in I$. In the second case, i.e., when $\ZZ_K(I)=\{{\bf 0}\}$, we obtain, by Theorem \ref{thm:finiteNull}, $I+\Gamma_q(k)=\I(\{{\bf 0}\})=\m$, which is homogeneous. 

For parts (3) and (4), the proofs follow analogously by replacing the field $k$ with $F$; the ideal $I$ with $J$; the homogeneous polynomials $h_1,\ldots,h_r$ in $S$ of degrees $d_1,\ldots,d_r$ with $t_1,\ldots,t_u$ in $F[X_0,\ldots,X_n]$ of degrees $s_1,\ldots,s_u$; $d$ with $s$; and the ideal $\mathfrak{d}$ with $\mathfrak{s}_F$, above. 
Note that the generating set of $J$ in $K[X_0,\ldots,X_n]$ is required to show that $\ell_j({\bf a})=0$ for every ${\bf a}\notin \V_K(J)$, which is a crucial step of the proof.
\end{proof}

\begin{cor}\label{thm:projFiniteNull2}
Let $I$ be a homogeneous ideal of $S$ generated by the homogeneous polynomials $h_1,\ldots,h_r$ of degrees $d_1,\ldots,d_r$, respectively. Define $d:=(d_1+\cdots+d_r)(q-1)+1$. 
Let $\mathfrak{d}$ denote the ideal generated by the $d$-th powers of $X_0,\ldots,X_n$ in $S$, i.e., $\mathfrak{d}=\left<X_0^d,\ldots,X_n^d\right>$. Then    
\begin{enumerate}
\item $\I(\V_{K}(I))=(I+\smash[b]{\Gamma_q^*}{(k)}):\mathfrak{d}=(I+\smash[b]{\Gamma_q^*}{(k)}):\mathfrak{m}^\infty=
\sqrt{I+\smash[b]{\Gamma_q^*}{(k)}}$ if $\V_K(I)$ is nonempty.
\item $\I(\V_K(I))=S$ if $\V_K(I)$ is empty. In this case, either $I=S$ or $I+\Gamma_q(k)=\m$.
\end{enumerate}
Furthermore, let $F$ be an algebraic extension of $K$ and let $J$ be a homogeneous ideal of $F[X_0,\ldots,X_n]$ generated by the homogeneous polynomials $t_1,\ldots,t_u$ of $K[X_0,\ldots,X_n]$ of degrees $s_1,\ldots,s_u$, respectively. If $s:=(s_1+\cdots+s_u)(q-1)+1$, then the vanishing ideal $\I_F(\V_K(J))$ considered as an ideal in $F[X_0,\ldots,X_n]$ satisfies  
\begin{enumerate}
\item[(3)] $\I_F(\V_{K}(J))=(J+\smash[b]{\Gamma_q^*}{(F)}):\mathfrak{s}_F=(J+\smash[b]{\Gamma_q^*}{(F)}):\mathfrak{m}_F^{\infty}=\sqrt{J+\smash[b]{\Gamma_q^*}{(F)}}$ if $\V_K(J)$ is nonempty. Here, $\mathfrak{m}_F$ and $\mathfrak{s}_F$ denote the ideals generated by $X_0,\ldots,X_n$ and $X_0^s,\ldots,X_n^s$ in $F[X_0,\ldots,X_n]$, respectively. 
\item[(4)] $\I_F(\V_K(J))=F[X_0,\ldots,X_n]$ if $\V_K(J)$ is empty. In this case, either the ideal $J=F[X_0,\ldots,X_n]$ or $J+\Gamma_q(F)=\m_F$. 
\end{enumerate}

\end{cor}
\begin{proof} 
For part (1), observe that all ideals in the required equalities are homogeneous. By 
Theorem~\ref{thm:projFiniteNull1}, along with the observation that $\sqrt{I+\smash[b]{\Gamma_q^*}{(k)}}\subseteq \I(\V_K(I))$ in $S$, it is enough to show that 
$$
(I+\smash[b]{\Gamma_q^*}{(k)}):\mathfrak{d}\subseteq (I+\smash[b]{\Gamma_q^*}{(k)}):\mathfrak{m}^\infty \subseteq 
\sqrt{I+\smash[b]{\Gamma_q^*}{(k)}}.
$$
For the first containment, note that $(I+\smash[b]{\Gamma_q^*}{(k)}):\mathfrak{d}\subseteq (I+\smash[b]{\Gamma_q^*}{(k)}):\mathfrak{m}^{(d-1)(n+1)+1}$. Indeed, if $f\in  (I+\smash[b]{\Gamma_q^*}{(k)}):\mathfrak{d}$, then $X_i^d f\in I+\smash[b]{\Gamma_q^*}{(k)}$ for $0\le i\le n$, and hence, by the pigeonhole principle, $X_0^{a_0}\cdots X_n^{a_n} f\in I+\smash[b]{\Gamma_q^*}{(k)}$ for any tuple $(a_0,\ldots,a_n)\in \Z_{\ge 0}$ satisfying $a_0+\cdots+a_n=(d-1)(n+1)+1$. For the second containment, let $g \in  (I+\smash[b]{\Gamma_q^*}{(k)}):\mathfrak{m}^\infty$ be a homogeneous polynomial. Without loss of generality, we can assume that $g$ is nonzero. Then $g$ is non-constant, since $\V_K(I)$ is nonempty. Now there exists $j\ge 1$ such that $X_0^{a_0}\cdots X_n^{a_n} g\in I+\smash[b]{\Gamma_q^*}{(k)}$ for any tuple $(a_0,\ldots,a_n)\in \Z_{\ge 0}$ satisfying $a_0+\cdots+a_n=j$. In particular, $X_i^j g\in I+\smash[b]{\Gamma_q^*}{(k)}$ for every $0\le i\le n$. Thus, $(X_i g)^j \in I+\smash[b]{\Gamma_q^*}{(k)}$ for  $0\le i\le n$. Therefore, $X_i g\in \sqrt{I+\smash[b]{\Gamma_q^*}{(k)}}$. Since $g$ is a non-constant homogeneous polynomial, $\mu g\in \sqrt{I+\smash[b]{\Gamma_q^*}{(k)}}$ for every nonzero term $\mu$ of $g$. Thus, $g^2\in \sqrt{I+\smash[b]{\Gamma_q^*}{(k)}}$, and hence, $g\in \sqrt{I+\smash[b]{\Gamma_q^*}{(k)}}$, as required.

For part (3), the proof follows analogously by replacing the field $k$ with $F$; the ideal $I$ with $J$; $d$ with $s$; and the ideal $\mathfrak{d}$ with $\mathfrak{s}_F$, above.

Parts $(2)$ and $(4)$ follow directly by Theorem~\ref{thm:projFiniteNull1}.  
\end{proof}

\section{Laksov-Westin conjectures on Nullstellensatz}

Laksov and Westin \cite{LW} introduced some conjectures on Nullstellensatz to strengthen the Hilbert $K$-Nullstellensatz. The strengthening is based on considering a smaller set in place of $P_K(m)$. Laksov in \cite[Section~5]{L} showed that the following set 
$$
P_K^0(m):=\{p\in k[y_0,\ldots,y_m]: p \text{ is homogeneous with only the trivial zero}\}
$$
will not work in place of $P_K(m)$ with counterexamples in both zero and nonzero characteristics. 
We define the following sets to state the conjectures of Laksov and Westin \cite[Section~8]{LW}.

\begin{enumerate}
\item $R_1=\{p(y_0^n,y_1,\ldots ,y_m): m,n\in \Z_{\ge 0} \text{ with }n\ge 1 \text{ and }p\in P^0_K(m)\}$.
\vspace{2mm}
\item $R_2=\{p(q(y_0,y_1,\ldots,y_n), y_{1+n},\ldots ,y_{m+n}): m,n\in \Z_{\ge 0}, p\in P^0_K(m), q\in P_K^0(n) \}$.
\vspace{1mm}
\vspace{-3mm}
\item $R_3=\{p_i(\cdots(p_2(p_1(y_0,y_1,\ldots,y_{m_1}),y_{m_1+1},\ldots,y_{m_2})\cdots),\ldots,y_{m_i}): i\ge 1, \\
\text{\hspace{10mm}} m_1,\ldots,m_i\in \Z_{\ge 0} \text{ with }m_1\le \cdots \le m_i, \text{ and } 
p_j\in P^0_K(m_j-m_{j-1}) \text{ for }\\
\text{\hspace{1cm}}   1\le j\le i \text{ with }m_0:=0\}$.
\vspace{2mm}
\item $R_4=\{p(y_0,y_1,\ldots,y_{m_1},y_{m_1+1},\ldots,y_{m_2},\ldots,y_{m_i}): i\ge 1, m_1,\ldots,m_i\in \Z_{\ge 0}, \text{ and} \\
\text{\hspace{1cm}$p$\text{ has }only the trivial zero and is quasi-homogeneous in the set of}  \\
\text{\hspace{1cm}variables } \{y_0,\ldots,y_{m_1}\},\{y_{m_1+1},\ldots,y_{m_2}\},\ldots,\{y_{m_{i-1}+1}, \ldots, y_{m_i}\}$ of \\
$\text{\hspace{1cm}\text{weights $w_1,w_2,\ldots,w_i$, respectively }where }w_1<w_2<\cdots<w_i\}$.
\end{enumerate}
Note that $R_1\subseteq R_2 \subseteq R_3 \subseteq R_4$. 

\begin{conjecture}[Laksov and Westin~\cite{LW}] \label{conjLakWes}
Denote by $R$ the polynomial ring $k[X_1,\ldots,X_n]$. Let $I$ be an ideal of $R$. For $1\le i\le 4$, let 
\begin{eqnarray*}
& R_i(I):=& \{f\in R: \text{ there exist }p\in R_i \text{ with }m\in \Z_{\ge 0},\text{ and } f_1,\ldots,f_m\in R \\
 & & \hspace{1.36cm}  \text{ such that } p(f,f_1,\ldots,f_m)\in I\}.
\end{eqnarray*} 

Then for any pair of fields $k\subseteq K$, $R_i(I)=\I(\ZZ_K(I))$ for $1\le i\le 4$.  
\end{conjecture}

\emph{Counterexample for the sets $R_1,R_2,R_3$.} We use the same example which Laksov in \ \
\cite[Section 5(1)]{L} considered to give a counterexample for the set $P_K^0(m)$. Let $k=K=\F_2$ and $R=\F_2[X_1,X_2]$. Consider the ideal $I=\left<X_1\right>$ in $R$. Then as observed in Remark \ref{rem:finiteCounter}, $\I(\ZZ_K(I))=\sqrt[\F_2]{I}=\left<X_1, X_2^2-X_2\right>$. Note that it suffices to show that $\I(\ZZ_K(I))\ne R_3(I)$ because $R_1\subseteq R_2 \subseteq R_3$. We claim that
$$
R_3(I)=\left<X_1\right>.
$$
Let $f\in R_3(I)$. Then there exists a polynomial $p\in R_3$ with an integer $m\ge 0$ and polynomials $f_1,\ldots, f_m\in R$ such that $p(f,f_1,\ldots,f_m)\in I$. That is, there exist $i\in \N$ and $m_1,\ldots,m_i\in \Z$ where $0\le m_1\le \ldots\le m_i=m$, and 
$p_j\in P^0_K(m_j-m_{j-1})$ for $1\le j\le i$ with $m_0:=0$ such that 
$$
p=p_i(\cdots(p_2(p_1(y_0,y_1,\ldots,y_{m_1}),y_{m_1+1},\ldots,y_{m_2})\cdots),y_{m_{i-1}+1},\ldots,y_{m_i}).
$$
Write 
$$
h_i:=p_{i-1}(\cdots (p_2(p_1(f,f_1,\ldots,f_{m_1}),f_{m_1+1},\ldots,f_{m_2})\cdots),\ldots,f_{m_{i-1}}).
$$
Then $p_i(h_i,f_{m_{i-1}+1},\ldots,f_{m_i})\in I$ where $p_i\in P_K^0(m_i-m_{i-1})$. We prove that all the polynomials $h_i,f_{m_{i-1}+1},\ldots,f_{m_i}\in I$. To simplify our notations, write $h_i(X_1,X_2):=\ell_0(X_1,X_2)$ and $f_{m_{i-1}+j}(X_1,X_2):=\ell_j(X_1,X_2)$ for $1\le j\le m_i-m_{i-1}$. Assume the contrary, i.e., $\ell_s\notin I$ for at least one $s\in \{0,\ldots,m\}$. Since $p_i(\ell_0,\ell_1,\ldots,\ell_{m_i-m_{i-1}})\in I$, the polynomial
$$
p_i(\ell_0(0,X_2),\ell_1(0,X_2),\ldots,\ell_{m_i-m_{i-1}}(0,X_2))
$$
is identically zero in $k[X_2]$. Write 
$$
\ell_j(0,X_2)=X_2^{d} g_j(X_2) \quad \text{for } j=0,\ldots,m_{i}-m_{i-1} 
$$
where $g_j\in k[X_2]$ and integer $d\ge 0$ is such that $X_2$ does not divide $g_e(X_2)$ for some $e$ (we can find such $e$ because $\ell_s(0,X_2)\ne 0$). Thus, $p_i(\ell_0(0,X_2),\ell_1(0,X_2),\ldots,\ell_{m_i-m_{i-1}}(0,X_2))$ is equal to
\begin{eqnarray*}
X_2^{d \deg p_i} p_i(g_0(X_2),g_1(X_2),\ldots,g_{m_i-m_{i-1}}(X_2))
\end{eqnarray*}
and is identically zero. Hence, $p_i(g_0(X_2),g_1(X_2),\ldots,g_{m_i-m_{i-1}}(X_2))$ is identically zero. But then $p_i(g_0(0),g_1(0),\ldots,g_{m_i-m_{i-1}}(0))=0$, which is not possible because $g_e(0)\ne 0$ and $p_i\in P^0_K(m_i-m_{i-1})$. Therefore, $h_i\in I$. Now write
$$
h_{i-1}:=p_{i-2}(\cdots (p_1(f,f_1,\ldots,f_{m_1})\cdots),f_{m_1+1},\ldots,f_{m_2})\cdots),\ldots,f_{m_{i-2}}).
$$
Then $h_{i}=p_{i-1}(h_{i-1},f_{m_{i-2}+1},\ldots,f_{m_{i-1}})\in I$. Since $p_{i-1}\in P^0_K(m_{i-1}-m_{i-2})$, following same steps as above, we obtain $h_{i-1}\in I$. Continuing in this way, we derive that $f\in I$. Thus, $R_3(I)=\left<X_1\right>\ne \I(\ZZ_K(I))$. $\hfill\square$

\vspace{3mm}
We are uncertain whether the above counterexample for $R_1,R_2$ and $R_3$ could also work for $R_4$ to falsify the conjecture. It would be interesting to investigate whether $R_4$ can serve as a substitute for $P_K(m)$ in the $K$-radical for the Nullstellensatz over arbitrary fields, or if it can be proven false. 

\begin{remark}
\rm{Conjecture~\ref{conjLakWes} by Laskov and Westin is motivated by the Nullstellensatz of Dubois~\cite{D} and Risler~\cite{Ri} for real closed fields, where they consider the set $P^0_K(m)$ for $m\ge 0$ with only the trivial zero. Following their Nullstellensatz, Tognoli questioned whether the set $P^0_K(m)$ works for an arbitrary field and it was generally believed that $P^0_K(m)$ will always work, until the work of Laksov. An interested reader can refer to Laksov and Westin~\cite{LW} for more details. }
\end{remark} 

\section*{Acknowledgements}
The author would like to thank her PhD advisor, Professor Sudhir R. Ghorpade, for his invaluable support during this work, especially for bringing reference~\cite{GRS} to her attention.

\end{document}